\documentclass[11pt,a4paper,reqno]{amsart}
\usepackage{mathrsfs,amssymb}

\usepackage{pstricks,pst-node,pst-plot}
\usepackage{epic}
\usepackage{autograph}

\newtheorem{theorem}{Theorem}

\newtheorem{corollary}{Corollary}
\newtheorem{definition}{Definition}
\newtheorem{proposition}{Proposition}
\newtheorem{lemma}{Lemma}

\newcommand{\ol}[1]{\overline{#1}}

\newcommand{\res}[3]{#1|^{#2}_{#3}}
\newcommand{\resmin}[4]{\res{#1}{#2}{#3} \setminus #4}

\begin{document}

\title[Free Adequate Semigroups]{Free Adequate Semigroups}

\keywords{adequate semigroup, free object, word problem, Munn tree}
\subjclass[2000]{20M10; 08A10, 08A50}
\maketitle

\begin{center}

    MARK KAMBITES

    \medskip

    School of Mathematics, \ University of Manchester, \\
    Manchester M13 9PL, \ England.

    \medskip

    \texttt{Mark.Kambites@manchester.ac.uk} \\
\end{center}

\begin{abstract}
We give an explicit description of the free objects in the quasivariety
of adequate semigroups, as sets of labelled
directed trees under a natural combinatorial multiplication. The 
morphisms of the free adequate semigroup onto the free ample semigroup and
into the free inverse semigroup are realised by a combinatorial ``folding''
operation which transforms our trees into Munn trees.
We use these results to show that free adequate semigroups and monoids
are $\mathcal{J}$-trivial and never finitely
generated as semigroups, and that those which are finitely
generated as $(2,1,1)$-algebras have decidable word problem.
\end{abstract}

\section{Introduction}

The structural theory of semigroups has traditionally been largely 
concerned with semigroups which admit local inverses with respect to 
non-identity idempotents; chief among these are the \textit{regular} and 
\textit{inverse} semigroups. In the late 1970's, it was observed by 
Fountain \cite{Fountain79,Fountain82} that many of the desirable 
properties of regular and inverse semigroups stem not directly from the 
existence of local inverses, but rather from the consequent fact that the 
cancellation properties of elements in such a semigroup are reflected in 
cancellation properties of the idempotents. This observation opened up 
effective methods of study for much wider classes of semigroups, the 
development of which forms the basis of the \textit{York School} of 
semigroup theory.

Amongst the classes of semigroups introduced by the York school,
the oldest and perhaps most natural is the class of
 \textit{adequate semigroups}.
 Adequate semigroups generalise inverse semigroups in something akin to the way that
cancellative monoids generalise groups; indeed the single-idempotent
adequate monoids are exactly the cancellative monoids, in exactly the
same way that the single-idempotent inverse monoids are exactly the
groups. They relate to \textit{abundant semigroups} \cite{Fountain82} in the same way
that inverse semigroups relate to \textit{regular semigroups}.
Adequate semigroups are most naturally viewed as algebras (in
the sense of universal algebra) of signature $(2,1,1)$, where the usual
multiplication is augmented with unary operations $*$ and $+$ which map
each element to certain idempotents which share its left and right
cancellativity properties (see Section~\ref{sec_preliminaries}
for a precise definition).
Within the category of $(2,1,1)$-algebras (although not in the category
of semigroups) the adequate semigroups form a quasivariety, and it
follows from general principles
(see, for example, \cite[Proposition~VI.4.5]{Cohn81}) that there exists a free adequate
semigroup for each cardinality of generating set.

When studying any class of algebras, it is extremely helpful to have
an explicit combinatorial description of the free objects in the class.
Such a description allows one for example to understand which identities
do and do not hold in the given class, and potentially to express any
member of the class as a collection of equivalence classes of elements
in a free algebra. In the case of inverse semigroups, a description of the free
objects first discovered by
Scheiblich \cite{Scheiblich72} was developed by Munn \cite{Munn74} into an
elegant geometric representation which has been of immense value in
understanding inverse semigroups.
Subsequently there have appeared a number of alternative proofs of Munn's
result and different representations for the free inverse
semigroup \cite{Gutierrez00,Petrich90,Poliakova05,Reynolds84}. Variants of Munn's
approach have since been used to describe the free objects in a number
of more general classes of semigroups \cite{Fountain88,Fountain91,Fountain07}
and also in the closely related setting of Cockett-Lack restriction categories
\cite{Cockett06}. All of these techniques rely on certain identities satisfied
by the classes of semigroups (or categories) in question, which permit the rewriting of
expressions to move idempotents to one side, and hence allow the systematic
decomposition of each element as a product of an idempotent part and an element of a
free subsemigroup. In a general adequate semigroup, by contrast, a typical
element cannot be written as such a product, so there is no hope of directly
applying Munn's technique, and hitherto no explicit description of the free
adequate semigroup has been found.

The main aim of this paper is to give an explicit geometric representation of
the free adequate semigroup on a given set, as a collection of edge-labelled directed
trees under a natural multiplication operation.
This result is inspired by Munn's celebrated characterisation of free inverse
semigroups as subtrees of the Cayley graph of the free group \cite{Munn74};
indeed the natural map from a free adequate semigroup to the free
inverse semigroup admits a natural interpretation as a ``folding'' map of
our trees onto Munn
trees. Partly for this reason, we believe that our representation is the
natural analogue of Munn's for adequate semigroups, and is likely to prove
correspondingly useful in the study of adequate semigroups.

As examples of how our main theorems can be applied, we show 
that every free adequate semigroup
or monoid is $\mathcal{J}$-trivial, while non-trivial examples are never
finitely generated as semigroups. Our 
representation also gives rise to a decision algorithm for the
word problem in these semigroups, although we do not claim that this
algorithm is tractable for large words. The computational complexity of
the word problem remains for now unclear, and deserves further study.

In a subsequent paper \cite{K_onesidedadequate}, we shall show that our
approach also leads to a description of the free objects in the categories
of left and right adequate
semigroups (roughly speaking, those semigroups which satisfy the
conditions defining adequate semigroups on one side only). An
alternative approach to free left and right adequate semigroups appears
in recent work of Branco, Gomes and Gould \cite{Branco09}.

In addition to this introduction, this article comprises six
sections. In Section~\ref{sec_preliminaries} we briefly recall the
definition of adequate semigroups, and summarise some of their
elementary properties on which we later rely. Sections~\ref{sec_trees}
and \ref{sec_algebra} introduce respectively the basic combinatorial objects from which our
representation is constructed, and the operations we shall use for combining
them, as well as proving some foundational results concerning them.
Section~\ref{sec_fas} is devoted to the proof that the resulting algebraic
structures are free objects in the quasivariety of adequate semigroups.
Section~\ref{sec_remarks} contains remarks on our
characterisation and its relationship with other work, while Section~\ref{sec_applications} shows how it can be
applied to establish some basic algebraic properties of the free adequate
semigroups and monoids.

\section{Preliminaries}\label{sec_preliminaries}

In this section we briefly recall some definitions, notation and terminology
relating to adequate semigroups. For a more comprehensive and detailed
introduction, see \cite{Fountain79}.

Recall that if $S$ is a semigroup without identity then $S^1$ denotes the monoid obtained
by adjoining an additional element to $S$ which acts as a multiplicative
identity element; if $S$ is already a monoid then we define $S^1 = S$. On any semigroup $S$, an equivalence relation $\mathcal{L}^*$ is
defined by $a \mathcal{L}^* b$ if and only if we have $ax = ay \iff bx = by$
for every $x, y \in S^1$. Dually, an equivalence relation $\mathcal{R}^*$ is
defined by $a \mathcal{R}^* b$ if and only if we have $xa = ya \iff xb = yb$
for every $x, y \in S^1$.

A semigroup is called \textit{left abundant} [\textit{right abundant}] if
every $\mathcal{R}^*$-class [respectively, every $\mathcal{L}^*$-class]
contains an idempotent. A semigroup is \textit{abundant} if it is both
left abundant and right abundant. If an abundant [left abundant, right
abundant] semigroup has the additional property that the idempotents commute,
then the semigroup is called \textit{adequate} [\textit{left adequate},
\textit{right adequate}]. It is
easily seen that, in a left [right] adequate semigroup, each $\mathcal{R}^*$-class
[$\mathcal{L}^*$-class] must contain a \textit{unique} idempotent.
We denote by $x^+$ [respectively, $x^*$] the unique idempotent in the
$\mathcal{R}^*$-class [respectively, $\mathcal{L}^*$-class] of an element
$x$; this idempotent acts as a left [right] identity for
$x$. The unary operations $x \mapsto x^+$ and $x \mapsto x^*$ are of such critical
importance in the theory of adequate [left adequate, right adequate]
semigroups that it is usual to consider these semigroups as algebras of signature
$(2,1,1)$ [or $(2,1)$ for left adequate and right adequate semigroups]
with these operations. In particular, one restricts attention to
morphisms which preserve the $+$ and/or $*$ operations (and hence
coarsen the 
$\mathcal{R}^*$ and $\mathcal{L}^*$ relations) as well as the
multiplication. These form a proper
subclass of the semigroup morphisms between the adequate semigroups, as
can be seen by considering for example any map from a free monoid (which
is cancellative and hence adequate) onto
any adequate monoid with more than one idempotent. Similarly, adequate
[left or right adequate] monoids may be viewed as algebras of signature
$(2,1,1,0)$ [$(2,1,0)$] with the identity a distinguished constant symbol.

We mention one important subclass of the adequate semigroups. An
adequate semigroup $S$ is called \textit{ample} (also known as \textit{type A})
if $ae = (ae)^+ a$ and $ea = a(ea)^*$ for all elements $a \in S$ and
idempotents $e \in S$.

We now establish some basic properties of left and right adequate
semigroups; these are well-known but since the proofs are very short we
include them in order to keep this article self-contained.
\begin{proposition}\label{prop_adequatebasics}
Let $S$ be a left adequate [respectively, right adequate] semigroup and
let $a, b, e, f \in S$ with $e$ and $f$ idempotent. Then
\begin{itemize}
\item[(i)] $e^+ = e$ [$e = e^*$];
\item[(ii)] $(ab)^+ = (ab^+)^+$ [$(ab)^* = (a^*b)^*$];
\item[(iii)] $a^+ a = a$ [$a a^* = a$];
\item[(iv)] $ea^+ = (ea)^+$ [$a^* e = (ae)^*]$;
\item[(v)] $a^+(ab)^+ = (ab)^+$ and [$(ab)^* a^* = (ab^*)$];
\item[(vi)] If $ef = f$ then $(ae)^+(af)^+ = (af)^+$ [$(ea)^* (fa)^* = (fa)^*$].
\end{itemize}
\end{proposition}
\begin{proof}
In each case we prove only the claim for left adequate semigroups, the
other being dual.
\begin{itemize}
\item[(i)] By definition $e^+$ is the unique idempotent in the
$\mathcal{R}^*$-class of $e$, which since $e$ is idempotent must be $e$ itself.
\item[(ii)] Since $ab \mathcal{R}^* (ab)^+$ we have $x(ab)^+ = y(ab)^+$ if and
only if $xab = yab$. Since $b \mathcal{R}^* b^+$ this is true if and only if
$xab^+ = yab^+$. And since $ab^+ \mathcal{R}^* (ab^+)^+$ this is true if and
only if $x(ab^+)^+ = y(ab^+)^+$. Thus, $(ab)^+ \mathcal{R}^* (ab^+)^+$. But
both are idempotent and each $\mathcal{R}^*$-class contains a unique idempotent,
so we must have $(ab)^+ = (ab^+)^+$.
\item[(iii)] Since $a^+ \mathcal{R}^* a$ and $a^+ a^+ = a^+ = 1 a^+$ we have $a^+ a = 1 a = a$.
\item[(iv)] Since idempotents commute we have $(ea^+)(ea^+) = ee a^+ a^+ = ea^+$,
that is, $ea^+$ is idempotent. Now using (i) and (ii) we have
$ea^+ = (e a^+)^+ = (ea^+)^+ = (ea)^+$. 
\item[(v)] Using (iv) and (iii) we have $a^+(ab)^+ = (a^+ab)^+ = (ab)^+$.
\item[(vi)] We have $(ae)^+ (af)^+ = (ae)^+ (aef)^+ = (aef)^+ = (af)^+$,
where the second equality is an application of part (i).
\end{itemize}
\end{proof}

Recall that an object $F$ in a concrete category $\mathcal{C}$ is called
\textit{free} on a subset $\Sigma \subseteq F$ if every function from $\Sigma$ to an
object $N$ in $\mathcal{C}$ extends uniquely to a morphism from $F$ to $N$.
The subset $\Sigma$ is called a \textit{free generating set} for $F$, and its
cardinality is the \textit{rank} of $F$. A free object
in a given category is uniquely determined up to isomorphism by its rank,
so it is usual to speak of \textit{the} free object of a given rank in a
given category.

Within the class of $(2,1,1)$-algebras, the adequate semigroups form a
quasivariety, defined by the quasi-identities
$AX = AY \iff A^*X = A^*Y$ and $XA = YA \iff XA^+ = YA^+$,
together with the associative law for multiplication and further
identities which ensure that the unary operations are idempotent and
have idempotent
and commutative images. A corresponding statement applies to the class
of adequate monoids.
Since every quasivariety contains free objects of every rank
(see, for example, \cite[Proposition~VI.4.5]{Cohn81}) it follows that there
exist free adequate semigroups and monoids of every rank. The chief aim of
the present paper
is to give an explicit geometric representation of these. We begin with a
proposition, the essence of which is that the distinction between semigroups
and monoids is unimportant.

\begin{proposition}\label{prop_monoidsemigroup}
Let $\Sigma$ be an alphabet. The free adequate monoid on $\Sigma$ is
isomorphic to the free adequate
 semigroup
on $\Sigma$ with a single adjoined element which is an identity for
multiplication and a fixed point for $*$ and $+$.
\end{proposition}
\begin{proof}
Let $S$ be the free adequate semigroup on $\Sigma$, and let $T = S \cup \lbrace 1 \rbrace$
be the monoid obtained by adjoining an element $1$ which is a multiplicative
identity and fixed point for $*$ and $+$. Then certainly $T$ is a monoid, and it is easily verified from
the definitions that $T$ is adequate. Now
suppose $f : \Sigma \to M$
is a map from $\Sigma$ to an adequate monoid.
Since $S$ is free on $\Sigma$ and the monoid $M$ is also an adequate
 semigroup, $f$
extends uniquely to a $(2,1,1)$-morphism $g : S \to M$. We define a map
$h : T \to M$ by $h(1) = 1$ and $h(x) = g(x)$ for all $x \in S$. Then it is easily verified
that $h$ is a $(2,1,1,0)$-morphism from $T$ to $M$ which
extends $f$. Moreover, $h$ is the unique morphism with this property, since any other
such morphism would restrict to another morphism from $S$ to $M$
extending $f$, contradicting the assumption that $S$ is free.
Thus, $T$ is free adequate monoid on $\Sigma$.
\end{proof}

We discuss briefly the relationship of abundant and adequate semigroups to
regular and inverse semigroups. Recall that Green's relation $\mathcal{L}$
[$\mathcal{R}$] is defined on any semigroup $S$ by $x \mathcal{L} y$ [$x \mathcal{R} y]$ if and
only if $x$ and $y$ generate the same principal left [right] ideal.
A semigroup is called \textit{regular} if every $\mathcal{R}$-class and
every $\mathcal{L}$-class contains an idempotent; a regular semigroup is
called \textit{inverse} if in addition the idempotents commute.
It can be shown \cite{Fountain79} that two elements of $S$ are $\mathcal{L}^*$-related
[respectively, $\mathcal{R}^*$-related] if and only if there is an embedding
of $S$ into another semigroup in which their images generate the same principal
left ideal [respectively, principal right ideal]. It follows that
$\mathcal{L}^*$ and $\mathcal{R}^*$ are coarsenings of $\mathcal{L}$ and
$\mathcal{R}$, and hence that every regular semigroup is abundant and
every inverse semigroup is adequate. In fact, it can be shown moreover that
every inverse semigroup is ample.

\section{Trees and Pruning}\label{sec_trees}

In this section we introduce the combinatorial objects which will form 
the elements of our representation of the free adequate semigroup.
The main objects of our study are \textit{labelled directed trees}, by which
we mean edge-labelled directed graphs whose underlying undirected graphs are
trees. Note that such graphs have the property that there is \textit{at most
one} directed path between any two vertices. If $e$ is an edge in such a
tree, we denote by $\alpha(e)$, $\omega(e)$ and $\lambda(e)$ the vertex at
which $e$ starts, the vertex at which $e$ ends and the label of $e$
respectively.

\begin{definition}[$\Sigma$-trees]
Let $\Sigma$ be an alphabet. A \textit{$\Sigma$-tree} (or just
a \textit{tree} if the alphabet $\Sigma$ is clear) is a directed
tree with edges labelled by elements of $\Sigma$, and with two distinguished
vertices (the \textit{start} vertex and the \textit{end} vertex) such that
there is a (possibly empty) directed path from the start vertex to the end vertex.

A tree with only one vertex is called \textit{trivial}, while a tree with
start vertex equal to its end vertex is called \textit{idempotent}. A
tree with a single edge and distinct start and end vertices is called a
\textit{base tree}; we identify each base tree with the label of its
edge.

In any tree, the
(necessarily unique) directed path from the start vertex to the end vertex is called
the \textit{trunk} of the tree; the vertices of the graph which lie on the
trunk (including the start and end vertices) are called \textit{trunk
vertices} and the edges which lie on the trunk are called \textit{trunk edges}.
If $X$ is a tree we write $\theta(X)$ for the set of trunk
edges of $X$.
\end{definition}
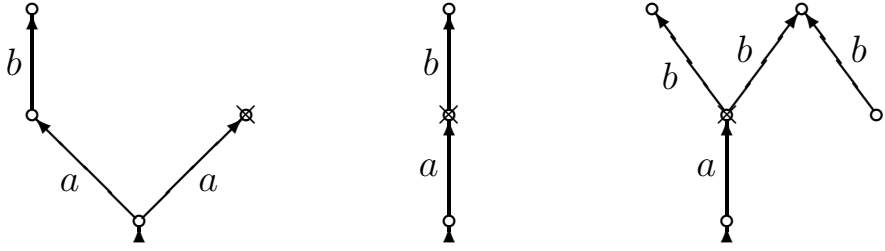
\begin{figure}\label{fig_first}
\begin{picture}(89,30)
\thicklines
\setloopdiam{10}
\Large
\setvertexdiam{1}

\letvertex A=(15,5)    \drawinitialstate[b](A){}
\letvertex B=(5,15)    \drawstate(B){}
\letvertex C=(5,25)    \drawstate(C){}
\letvertex D=(25,15)    \drawstate(D){$\times$}
\drawedge(A,B){$a$}
\drawedge(B,C){$b$}
\drawedge[r](A,D){$a$}

\letvertex E=(44,5)    \drawinitialstate[b](E){}
\letvertex F=(44,15)    \drawstate(F){$\times$}
\letvertex G=(44,25)    \drawstate(G){}
\drawedge(E,F){$a$}
\drawedge(F,G){$b$}

\letvertex H=(70,5)     \drawinitialstate[b](H){}
\letvertex I=(70,15)    \drawstate(I){$\times$}
\letvertex J=(63,25)    \drawstate(J){}
\letvertex K=(77,25)    \drawstate(K){}
\letvertex L=(84,15)    \drawstate(L){}
\drawedge(H,I){$a$}
\drawedge(I,J){$b$}
\drawedge(I,K){$b$}
\drawedge[r](L,K){$b$}

\end{picture}
\caption{Some examples of $\lbrace a, b \rbrace$-trees.}
\end{figure}

Figure~1 shows three examples of $\Sigma$-trees
where $\Sigma = \lbrace a, b \rbrace$. In each case, the start vertex is
marked by an arrow-head, and the end vertex by a cross. Notice that
a vertex may have multiple edges coming in or going out with the same label,
and that in each case there is a directed path (in our examples, a single
edge) from the start vertex to the end vertex.
Figure~2 shows the trivial tree and the base trees $a$ and
$b$.

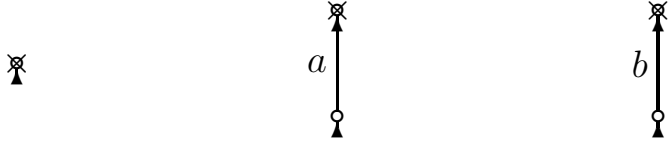
\begin{figure}\label{fig_sesqui}
\begin{picture}(70,20)
\thicklines
\setloopdiam{10}
\Large
\setvertexdiam{1}

\letvertex A=(5,10)    \drawinitialstate[b](A){$\times$}

\letvertex B=(35,5)    \drawinitialstate[b](B){}
\letvertex C=(35,15)    \drawstate(C){$\times$}
\drawedge(B,C){$a$}

\letvertex D=(65,5)    \drawinitialstate[b](D){}
\letvertex E=(65,15)    \drawstate(E){$\times$}
\drawedge(D,E){$b$}

\end{picture}
\caption{The trivial tree, the base tree $a$ and the base tree $b$.}
\end{figure}

\begin{definition}[Subtrees and morphisms]
Let $X$ and $Y$ be trees. A \textit{subtree} of $X$ is a subgraph
of $X$ containing the start and end vertices, the underlying undirected
graph of which is connected.

A \textit{morphism} $\rho : X \to Y$ of $\Sigma$-trees is a map taking
edges to edges and vertices to vertices, such that $\rho(\alpha(e)) = \alpha(\rho(e))$,
 $\rho(\omega(e)) = \omega(\rho(e))$ and $\lambda(e) = \lambda(\rho(e))$ for
all edges $e$ in $X$, and which maps the start and end vertex of $X$ to the
start and end vertex of $Y$ respectively. An
\textit{isomorphism} is a morphism which is bijective on both edges and
vertices.
\end{definition}

In our example of Figure~1, there is clearly a
unique morphism from the left-hand tree to the middle tree, and a unique
morphism taking the right-hand tree to the middle tree.  

It is easily shown that morphisms have the expected properties that the
composition of two morphisms (where defined) is again a morphism, while the restriction of a morphism to a
subtree is also a morphism. It is also easily verified (using the fact that
we consider only trees with a directed path from the start to the end vertex)
that a morphism necessarily maps the trunk edges of its domain bijectively
onto the trunk edges of its image.
Note that morphisms map $\Sigma$-trees to $\Sigma$-trees (for the same
alphabet $\Sigma$), and preserve the labelling of edges.

\begin{definition}
The set
of all isomorphism types of $\Sigma$-trees is denoted $UT^1(\Sigma)$
while the set of isomorphism types of non-trivial $\Sigma$-trees is
denoted $UT(\Sigma)$. The set of isomorphism types of idempotent
trees is denoted $UE^1(\Sigma)$, while the set of isomorphism types of
non-trivial idempotent trees is denoted $UE(\Sigma)$.
\end{definition}

Much of the time we shall be formally concerned not with trees themselves
but rather with isomorphism types. However, where no
confusion is likely, we shall for the sake of conciseness ignore the
distinction and implicitly identify trees with their respective isomorphism
types.

\begin{definition}[Retracts]
A \textit{retraction} of a tree $X$ is an idempotent morphism from
$X$ to $X$; its image is called a \textit{retract} of
$X$. A tree $X$ is called
\textit{pruned} if it does not admit a
non-identity retraction. The set of all isomorphism types of pruned trees
[respectively, non-trivial pruned trees] is denoted $T^1(\Sigma)$
[respectively, $T(\Sigma)$].
\end{definition}

Returning to our examples from Figure~1, neither the left-hand
nor middle tree admits any non-identity retraction, so these trees are pruned.
The right-hand tree admits four non-identity retractions. Note also that the
trivial tree and the base trees, examples of which are shown in Figure~2,
do not admit any non-identity retractions, and so are pruned trees.

Just as with morphisms, it is readily verified that a composition of
retractions (where defined) is a retraction, and the restriction of a retraction
to a subtree is again a retraction. The following proposition is an instance
of a well-known phenomenon, another important example of which is the uniqueness
of the \textit{core} of a finite graph \cite{Harary67,Hell92}. It can be deduced from
very general results about morphisms of finite relational structures (for example
\cite[Proposition~1.4.7]{Foniok07}) but for completeness we sketch a
simple combinatorial proof.

\begin{proposition}\label{prop_pruningconfluent}[Confluence of retracts]
For each tree $X$ there is a unique (up to isomorphism) pruned tree which
is a retract of $X$.
\end{proposition}
\begin{proof}
Clearly since $X$ is finite it has a retract of
minimal size, which must be a pruned tree. Suppose now that
$\pi_1 : X \to X$ and $\pi_2 : X \to X$ are both retractions with
pruned images $Y$ and $Z$ respectively. Consider the two
compositions $\pi_1 \pi_2$ and $\pi_2 \pi_1$; since they are maps
on a finite set we may choose a positive integer $n$ such that $(\pi_1 \pi_2)^n$
and $(\pi_2 \pi_1)^n$ are idempotent morphisms, that is, retractions of $X$.
A straightforward argument shows
that the restriction of $(\pi_1 \pi_2)^n$ to $Y$ is a retraction of $Y$, which
since $Y$ is pruned means it must be the identity map on $Y$. Dually, the restriction
of $(\pi_2 \pi_1)^n$ to $Z$ is the identity map on $Z$. It now follows easily
that $\pi_1$ and $\pi_2(\pi_1 \pi_2)^{n-1}$ restrict to give mutually inverse
isomorphisms between $Y$ and $Z$, as required.
\end{proof}

Proposition~\ref{prop_pruningconfluent} explains why we must
formally work with isomorphism types of trees rather than trees themselves:
choices made during the process of ``pruning'' may result in distinct but
isomorphic pruned trees, and it is necessary that we view these as
the same object.

\begin{definition}[Pruning of a tree]
Let $X \in UT^1(\Sigma)$. Then the \textit{pruning} of $X$ is the unique
(by Proposition~\ref{prop_pruningconfluent}) element of $T^1(\Sigma)$
which can be obtained from $X$ by pruning. It is denoted $\ol{X}$.
\end{definition}

Considering again our examples from Figure~1, we have already
seen that the left-hand and middle trees are pruned, and so each is (or
more properly, the isomorphism type of each is) its own pruning. Two of
the four retracts of the right-hand tree are pruned; Proposition~\ref{prop_pruningconfluent} tells us that
these must be isomorphic, and indeed they are both isomorphic to the middle
tree of Figure~1. Hence, the pruning of the right-hand tree
is (the isomorphism type of) the middle tree.

\section{Algebra on Trees}\label{sec_algebra}

We now define some operations on isomorphism types of trees.

\begin{definition}[Unpruned operations]
We define a product operation, called \textit{unpruned multiplication}, on
$UT^1(\Sigma)$ as follows. For $X, Y \in UT^1(\Sigma)$, choose representative
$\Sigma$-trees $X'$ for $X$ and $Y'$ for $Y$ such that $X' \cap Y' = \lbrace v \rbrace$ where
$v$ is the end vertex of $X'$ and the start vertex of $Y'$. Then the
\textit{unpruned product} $X \times Y$ is the isomorphism type of the
tree with graph $X' \cup Y'$ (with the maps $\alpha$, $\omega$ and $\lambda$
extending the corresponding maps in $X'$ and $Y'$) considered as a
$\Sigma$-tree with start vertex the start vertex of $X'$ and end vertex the
end vertex of $Y'$.

We also define two unary operations on $UT^1(\Sigma)$, called
\textit{unpruned $(+)$} and \textit{unpruned $(*)$}. If $X'$ is a representative
$\Sigma$-tree for $X \in UT^1(\Sigma)$ then $X^{(+)}$ is the isomorphism type
of the idempotent tree with the same underlying graph
and start vertex as $X'$, but with end vertex the start vertex of $X'$. 
Dually, $X^{(*)}$ is the isomorphism type of the idempotent tree with the
same underlying graph and end
vertex as $X'$, but with start vertex the end vertex of $X'$.
\end{definition}

The above definition is rendered rather technical by the formal need to work
with representatives of isomorphism types, but intuitively the
operations are very simple. For example, unpruned multiplication
simply means ``gluing together'' two trees by identifying the end vertex 
of one with the start vertex of the other. The following proposition 
gives some elementary properties of these operations.

\begin{proposition}\label{prop_unprunedproperties}
Unpruned multiplication is an associative binary operation
on the set $UT^1(\Sigma)$ of isomorphism types of $\Sigma$-trees. The isomorphism type of the trivial tree is an identity element for this operation. The set
$UT(\Sigma)$ of isomorphism types of non-trivial trees forms a subsemigroup
of $UT^1(\Sigma)$.

The maps
$X \mapsto X^{(+)}$ and $X \mapsto X^{(*)}$ are 
idempotent unary operations on $UT^1(\Sigma)$. The subsemigroup
generated by the images of these unary operations is commutative.
\end{proposition}
\begin{proof}
It is easily seen that unpruned multiplication is associative, that
the trivial tree acts as an identity, and that the product of two non-trivial
trees is never trivial. Finally, the images of the $(+)$ and $(*)$ operations
are by definition idempotent trees, and it is immediate from the definitions
that multiplication of idempotent trees is commutative.
\end{proof}

\begin{definition}[Pruned operations]
Let $X$ and $Y$ be isomorphism types of pruned trees. Then we define $XY = \ol{X \times Y}$,
$X^* = \ol{X^{(*)}}$ and $X^+ = \ol{X^{(+)}}$.
\end{definition}

Returning to our example trees from Figure~1, and recalling that we
identify the letters $a$ and $b$ with the corresponding base trees (as
shown in Figure~2), we see that the trees depicted
correspond to the unpruned expressions $(a \times b)^{(+)} \times a$,
$a \times b^{(+)}$ and $a \times b^{(+)} \times (b \times b^{(*)})^{(+)}$
respectively.

\begin{proposition}\label{prop_prunedproperties}
Pruned multiplication is a well-defined binary operation on the set
$T^1(\Sigma)$ of isomorphism types of pruned trees. The unary operations
$*$ and $+$ are well-defined idempotent unary operations on the set $T^1(\Sigma)$
of isomorphism types of pruned trees.
\end{proposition}
\begin{proof}
The claims follow easily from Proposition~\ref{prop_unprunedproperties}.
\end{proof}

We are now ready to prove a basic but important foundational result.
\begin{theorem}\label{thm_morphism}
The pruning map
$$UT^1(\Sigma) \to T^1(\Sigma), \ X \mapsto \ol{X}$$
is a surjective $(2,1,1,0)$-morphism from the set of isomorphism types of
$\Sigma$-trees under unpruned multiplication, unpruned $(*)$ and unpruned $(+)$ with distinguished
identity element to
the set of isomorphism types of pruned trees under pruned multiplication,
$*$ and $+$ with distinguished identity element.
\end{theorem}
\begin{proof}
First notice that every isomorphism type $X \in T^1(\Sigma)$ of pruned
trees is also an isomorphism type of $\Sigma$-trees and satisfies
$\ol{X} = X$; thus, the given map is surjective.

Now let $X$ and $Y$ be unpruned trees. Let $\pi_X : X \to X$ and
$\pi_Y : Y \to Y$ be retractions with images $\ol{X}$
and $\ol{Y}$ respectively.

We show first that $\ol{X} \ \ol{Y} = \ol{X \times Y}$. 
 Notice that since the amalgamated vertex in
the unpruned product $X \times Y$ is the end vertex of $X$ and the start
vertex of $Y$, it is fixed by both
$\pi_X$ and $\pi_Y$; it follows that there is a (necessarily unique, since
every vertex and edge of $X \times Y$ comes from $X$ or $Y$) map
$\pi : X \times Y \to X \times Y$ which extends both $\pi_X$ and $\pi_Y$.
Clearly, $\pi$ is a morphism. Since
$\pi_X$ and $\pi_Y$ are idempotent and at least one of them is defined on
each vertex and edge of $X \times Y$, we see also that $\pi$ is idempotent,
and hence is a retraction. Moreover, it follows immediately
from the definition of unpruned multiplication that $\pi(X \times Y) = \ol{X} \times \ol{Y}$.
But now by Proposition~\ref{prop_pruningconfluent} we have
$$\ol{X \times Y} = \ol{\pi(X \times Y)} = \ol{\ol{X} \times \ol{Y}} = \ol{X} \ \ol{Y}.$$

Next we claim that $\ol{X}^+ = \ol{X^{(+)}}$. First notice that, since
$X^{(+)}$ has the same underlying labelled directed graph as $X$, the
same start vertex, and end vertex the start vertex of $X$, the retraction
$\pi_X$ of $X$ is also a retraction of $X^{(+)}$.
Clearly its image is the tree $\ol{X}^{(+)}$.
Hence by Proposition~\ref{prop_pruningconfluent} again we have
$$\ol{X^{(+)}} = \ol{\pi_X(X^{(+)})} = \ol{\ol{X}^{(+)}} = \ol{X}^+$$
as required.
A dual argument shows that $\ol{X}^* = \ol{X^{(*)}}$. Finally, we have
shown that pruning is a surjective semigroup morphism of monoids, so it
must preserve the identity. Thus, it is a $(2,1,1,0)$-morphism.
\end{proof}

From Theorem~\ref{thm_morphism} we deduce immediately that the
$(2,1,1,0)$-algebra $T^1(\Sigma)$ inherits a number of properties which
were obvious in $UT^1(\Sigma)$ but perhaps less so in $T^1(\Sigma)$.
\begin{corollary}\label{cor_prunedproperties}
Pruned multiplication is an associative operation on $T^1(\Sigma)$. The
unary operations $*$ and $+$ are idempotent; the subsemigroup generated
by their images is commutative.
\end{corollary}

As well as providing a theoretical underpinning for what we wish to do,
Theorem~\ref{thm_morphism} is extremely useful for computational purposes; it
means that complex expressions involving pruned trees in $T^1(\Sigma)$ can
be computed by first evaluating them in $UT^1(\Sigma)$ using unpruned
operations and then pruning the resulting tree only at the end. Since pruning is
the hardest part of such a computation, this can result in significant
efficiency savings.

Figure~3 shows some more examples of $\lbrace a, b \rbrace$-trees,
namely the elements of $UT^1(\lbrace a, b \rbrace)$
corresponding to the unpruned expressions
$$(a \times (b^{(+)} \times a)^{(*)})^{(+)} \times b \ \ \text{ and } \ \ a^{(+)} \times b$$
respectively. Notice that the right-hand tree is pruned, while left-hand tree
admits a retract isomorphic to the right-hand tree. This means
that we have
$$(a (b^+ a)^*)^+ b = 
\ol{(a \times (b^{(+)} \times a)^{(*)})^{(+)} \times b}
= \ol{a^{(+)} \times b}
= a^+ b$$
in the monoid $T^1(\lbrace a, b \rbrace)$. We shall see later that
$T^1(\lbrace a, b \rbrace)$ is actually a free adequate monoid, freely
generated by the base trees, so it follows
that the identity $(A(B^+ A)^*)^+ B = A^+ B$ holds in every adequate
monoid and semigroup. (The reader may wish to try verifying this directly
from the axioms for adequate semigroups.)

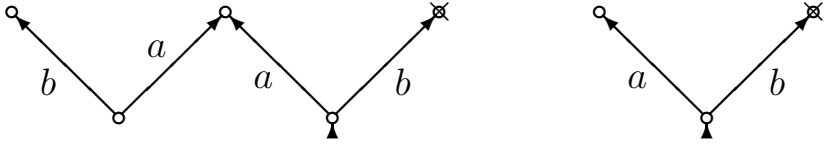
\begin{figure}\label{fig_second}
\begin{picture}(95,20)
\thicklines
\setloopdiam{10}
\Large
\setvertexdiam{1}

\letvertex A=(35,5)    \drawinitialstate[b](A){}
\letvertex B=(45,15)   \drawstate(B){$\times$}
\letvertex C=(25,15)   \drawstate(C){}
\letvertex D=(15,5)    \drawstate(D){}
\letvertex E=(5,15)   \drawstate(E){}
\drawedge[r](A,B){$b$}
\drawedge(A,C){$a$}
\drawedge(D,C){$a$}
\drawedge(D,E){$b$}

\letvertex F=(70,5)    \drawinitialstate[b](F){}
\letvertex G=(80,15)   \drawstate(G){$\times$}
\letvertex H=(60,15)   \drawstate(H){}
\drawedge[r](F,G){$b$}
\drawedge(F,H){$a$}
\end{picture}
\caption{The trees $(a \times (b^{(+)} \times a)^{(*)})^{(+)} \times b$ and $a^{(+)} \times b = a^+ b$
respectively.}
\end{figure}

Our next objective is to establish some algebraic properties of the
$(2,1,1,0)$-algebra $T^1(\Sigma)$ of pruned trees over a given alphabet
$\Sigma$.

\begin{proposition}\label{prop_idempotentsfix}
Let $X \in T^1(\Sigma)$. Then $X^+X = X = X X^*$.
\end{proposition}
\begin{proof}
We prove that $X^+ X = X$, the claim that $X X^* = X$ being dual.
By Theorem~\ref{thm_morphism} we have
$$X^+ X = \ol{X^{(+)} \times X}.$$
Consider the unpruned product $X^{(+)} \times X$. It follows straight
from the definitions of unpruned operations that this
consists of two copies ($X_1$ and $X_2$ say) of the tree $X$, with their
start vertices identified, and with start vertex this start vertex and
end vertex the end vertex of $X_1$. Define a map
$$\pi : X^{(+)} \times X \to X^{(+)} \times X$$
which fixes $X_1$ and maps each edge [vertex] of $X_2$ onto the corresponding
edge [vertex] of $X_1$. Then $\pi$ is a retraction of
$X^{(+)} \times X$ with image $X_1$ which is isomorphic to $X$.
Hence by Proposition~\ref{prop_pruningconfluent} we have
$$\ol{X^{(+)} \times X}
 = \ol{\pi(X^{(+)} \times X)}
 = \ol{X}.$$
\end{proof}

The following proposition justifies the name we have given to idempotent
trees.
\begin{proposition}\label{prop_idempotents}
For any $X \in T^1(\Sigma)$ the following are equivalent:
\begin{itemize}
\item[(i)] $X$ is an idempotent tree;
\item[(ii)] $X$ is an idempotent element under pruned multiplication;
\item[(iii)] $X = X^+$;
\item[(iv)] $X = Y^+$ for some $Y \in T^1(\Sigma)$;
\item[(v)] $X = X^*$;
\item[(vi)] $X = Y^*$ for some $Y \in T^1(\Sigma)$.
\end{itemize}
\end{proposition}
\begin{proof}
We prove the equivalence of (i), (ii), (iii) and (iv), the equivalence of
(i), (ii), (v) and (vi) being dual. That
(iii) implies (iv) is immediate, while (iv) implies (i) and (i) implies
(iii) follow directly from the definition of the $+$ operation.
If (iii) holds, so that $X = X^+$, then by Proposition~\ref{prop_idempotentsfix} we have
$$XX = X^+ X = X$$
so that (ii) holds.

To 
complete the proof we shall show that (ii) implies (i).
Indeed, suppose for a contradiction that (ii) holds, that is, that $X$ is
idempotent under pruned multiplication, but that (i) does not, so that
$X$ has distinct start and end vertices, and hence at least one trunk
edge. Let $n$ be the number of trunk edges in $X$. Then $X \times X$
has $2n$ trunk edges and, since pruning fixes the trunk, so does
$XX = \ol{X \times X}$. But since $n > 0$ we have $2n \neq n$, so this
contradicts the fact that $X$ is idempotent under pruned multiplication.
This completes the proof that (ii) implies (i).
\end{proof}

The following proposition shows that each pruned tree $X \in T(\Sigma)$ is
$\mathcal{L}^*$-related [respectively, $\mathcal{R}^*$-related] to the
idempotent $X^*$ [respectively, $X^+$], and hence that $T^1(\Sigma)$ is
abundant.

\begin{proposition}\label{prop_abundant}
Let $A, B, X \in T^1(\Sigma)$. Then
\begin{itemize}
\item $AX = BX$ if and only if $A X^+ = B X^+$; and
\item $XA = XB$ if and only if $X^* A = X^* B$.
\end{itemize}
\end{proposition}
\begin{proof}
We show that $AX = BX$ if and only if $A X^+ = B X^+$, the other claim being
dual. Certainly if $A$, $B$ and $X$ are pruned trees such that $AX^+ = BX^+$
then by Proposition~\ref{prop_idempotentsfix} we have
$$AX = A(X^+X) = (AX^+) X = (BX^+) X = B (X^+ X) = BX.$$

Conversely, suppose that $A$, $B$ and $X$ are pruned trees such that
$AX = BX$; we must show that $AX^+ = BX^+$.
Let $\pi_A : A \times X \to A \times X$ and $\pi_B : B \times X \to B \times X$
be retractions with images isomorphic to
$$\ol{A \times X} = AX = BX = \ol{B \times X}.$$
Now by the definition of unpruned operations, the tree
$A \times X^{(+)}$ has the same underlying graph as $A \times X$
and the same start vertex, but with end vertex at the start vertex of $X$
instead of the end vertex of $X$. Since the latter vertex lies in the trunk
of $A \times X$ it is fixed by $\pi_A$, and it follows that $\pi_A$ also defines a
retraction of $A \times X^{(+)}$; its image has the same underlying
graph as $AX = BX$ but with end vertex moved to the start vertex of $X$.
Similarly, $\pi_B$ defines a retraction of $B \times X^{(+)}$; its image
also has the same underlying graph as $AX = BX$ but with end vertex at the
start vertex of $X$. Now using Proposition~\ref{prop_pruningconfluent} and
Theorem~\ref{thm_morphism} we
have
$$A X^+ = \ol{A \times X^{(+)}} =
 \ol{\pi_A(A \times X^{(+)})} = \ol{\pi_B(B \times X^{(+)})} = 
 \ol{B \times X^{(+)}} = BX^+.$$
\end{proof}

We are now ready to prove our second main theorem.

\begin{theorem}\label{thm_adequate}
Let $\Sigma$ be an alphabet. Any subset of $T^1(\Sigma)$ closed under the
operations
of pruned multiplication, $+$ and $*$ forms an adequate semigroup under
these operations. Any subset
of $T^1(\Sigma)$ closed under the operations of pruned
multiplication and $+$ [respectively, $*$] forms a left adequate
[respectively, right adequate] semigroup under these operations.
\end{theorem}
\begin{proof}
Let $S$ be a subset of $T^1(\Sigma)$ closed under pruned multiplication and
$+$. Then for any element $X \in S$, by Proposition~\ref{prop_abundant}
we have $X \mathcal{R}^* X^+$, where by Proposition~\ref{prop_idempotents}
the element $X^+$ is idempotent under pruned multiplication; hence $S$ is
left abundant. Moreover, by Proposition~\ref{prop_idempotents} again, 
the idempotents under pruned multiplication in $S$ are exactly the elements
of the form $X^+$, which by Corollary~\ref{cor_prunedproperties} commute.
Thus, $S$ is left adequate.

A dual argument shows that if $S$ is a subset of $T^1(\Sigma)$ closed under
pruned multiplication and $*$ then $S$ is right adequate, and it follows
that if $S$ is closed under pruned multiplication, $+$ and $*$ then $S$
is adequate.
\end{proof}

\section{The Free Adequate Monoid and Semigroup}\label{sec_fas}

We saw in the previous section that, for any alphabet $\Sigma$, the
$(2,1,1,0)$-algebra $T^1(\Sigma)$ is an adequate monoid. In this section
we shall show that it is in fact a free adequate monoid, freely generated
by the subset $\Sigma$ of base trees. By Proposition~\ref{prop_monoidsemigroup} this
also establishes that $T(\Sigma)$ is the free adequate semigroup on
$\Sigma$.

To keep the proofs in this and the following sections concise, we shall
need some additional notation. If $X$ is a tree and $S$ is a
set of non-trunk edges and vertices of $X$ then $X \setminus S$ denotes
the largest subtree of $X$ (recalling that a subtree must be connected
and contain the start and end vertices, and hence the trunk) which does not
contain any vertices or edges from $S$.
If $s$ is a single edge or vertex we write $X \setminus s$ for $X \setminus \lbrace s \rbrace$.
If $u$ and $v$ are vertices of $X$ such that there is a directed path from
$u$ to $v$ then we shall denote by $\res{X}{u}{v}$ the tree which has the
same underlying labelled directed graph as $X$ but start vertex $u$ and end vertex $v$.
If $X$ has start vertex $a$ and end vertex $b$ then we define
$\res{X}{u}{} = \res{X}{u}{b}$ and $\res{X}{}{v} = \res{X}{a}{v}$ where
applicable. Thus, for example, $\resmin{X}{v}{v}{e}$ means the largest
connected subgraph of $X$ containing the vertex $v$ but not the edge $e$,
viewed as an idempotent tree with start and end vertex $v$.

\begin{proposition}\label{prop_generators}
The set $T^1(\Sigma)$ of pruned trees is generated as a $(2,1,1,0)$-algebra
by the set $\Sigma$ of base trees.
\end{proposition}
\begin{proof}
Let $\langle \Sigma \rangle$ denote the $(2,1,1,0)$-subalgebra of $T^1(\Sigma)$
generated by $\Sigma$. We wish to show that every tree in $T^1(\Sigma)$
is contained in $\langle \Sigma \rangle$.
We proceed by induction on number of edges. The tree with no
edges is the identity element of $T^1(\Sigma)$ and so by definition is
contained in every
$(2,1,1,0)$-subalgebra of $T^1(\Sigma)$, and in particular in $\langle \Sigma \rangle$.
Now suppose for induction that $X \in T^1(\Sigma)$ has at least one edge, and
that every tree in $T^1(\Sigma)$ with strictly fewer edges lies in $\langle \Sigma \rangle$.

First suppose that $X$ has at least one trunk edge. Let $v_0$ be the start
vertex of $X$, $e$ be the trunk edge incident with $v_0$, $a$ be the
label of $e$ and $v_1$ be the vertex at the end of
$e$, that is, the second trunk vertex of $X$. Let
$Y = \resmin{X}{v_0}{v_0}{e}$ and $Z = \resmin{X}{v_1}{}{e}$.
Then $Y$ and $Z$ are pruned trees
with strictly fewer edges than $X$, and so by induction lie in
$\langle \Sigma \rangle$. Now clearly from the definitions we have
$Y \times a \times Z = X$, and since $X$ is pruned using
Theorem~\ref{thm_morphism} we have
$$Y a Z = \ol{Y \times a \times Z} = \ol{X} = X$$
so that $X \in \langle \Sigma \rangle$ as required.

Next suppose that $X$ has no trunk edges. Let $e$ be any edge
incident with the start vertex $v_0$, and suppose $e$ has label
$a$. Let $v_1$ be the vertex at the other end of $e$ from $v_0$.
Define
$Y = \resmin{X}{v_0}{v_0}{e}$
and
$Z = \resmin{X}{v_1}{v_1}{e}$.
Then $Y$ and $Z$ are pruned trees
with strictly fewer edges than $X$, and so by induction lie in
$\langle \Sigma \rangle$.
Suppose first that $e$ is orientated away from the start vertex of $X$. Then
from the definitions of unpruned operations we have
$X = Y \times (a \times Z)^{(+)}$, and since $X$ is pruned applying
Theorem~\ref{thm_morphism} yields
$$Y (a Z)^+ = \ol{Y \times (a \times Z)^{(+)}} = \ol{X} = X.$$
In the case where $e$ is orientated towards the start vertex, a dual
argument yields $(Z a)^* Y = X$, thus establishing in all cases
that $X \in \langle \Sigma \rangle$.
\end{proof}

Now suppose $M$ is an adequate monoid and $\chi : \Sigma \to M$ is a function.
Recall that $\Sigma$ is identified with the set of (isomorphism types of) base
trees in $UT^1(\Sigma)$ and $T^1(\Sigma)$.
Our objective is to show that there is a unique $(2,1,1,0)$-morphism from
$T^1(\Sigma)$ to $M$ which extends the function $\chi$.

Let $E(M)$ denote the semilattice of idempotents of the adequate monoid $M$.
We begin by defining a map $\tau : UE^1(\Sigma) \to E(M)$ from the set
$UE^1(\Sigma)$ of (not necessarily pruned) idempotent trees to $E(M)$.
Let $X$ be an idempotent tree. If $X$ has
no edges then we define $\tau(X) = 1$. Otherwise, we define $\tau(X)$
recursively, in terms of the value of $\tau$ on idempotent trees with
strictly fewer edges than $X$, as follows.

Let $v$ be the start vertex of $X$ (which since $X$ is an idempotent tree
is also the end vertex of $X$).
Let $E^+(X)$ be the set of edges in $X$ which start at $v$, and 
$E^-(X)$ the set of edges in $X$ which end at $v$. Now we define
{\small \begin{align*}
\tau(X) = \left( \prod_{e \in E^+(X)}
[\chi(\lambda(e)) \tau(\resmin{X}{\omega(e)}{\omega(e)}{e})]^+ \right)
\left( \prod_{e \in E^-(X)}
[\tau(\resmin{X}{\alpha(e)}{\alpha(e)}{e}) \chi(\lambda(e))]^* \right).
\end{align*}}
Notice that since $X$ has at least one edge and is connected, this product
cannot be empty. Notice also that, since all the factors are idempotent and
idempotents commute in the adequate semigroup $M$, the value of this product is
idempotent, and is independent of the order in which the factors are multiplied.
The value clearly depends only on the isomorphism type of $X$, and so the
map $\tau$
is a well-defined function from $UE^1(\Sigma)$ to $E(M)$. We now establish
some of its elementary properties.

\begin{proposition}\label{prop_tausplit}
Let $X$ be an idempotent tree with start (and end) vertex $v$, and let $X_1$ and
$X_2$ be subtrees of $X$ such that $X = X_1 \cup X_2$ and $X_1 \cap X_2 = \lbrace v \rbrace$. Then $\tau(X) = \tau(X_1) \tau(X_2)$.
\end{proposition}
\begin{proof}
Clearly we have $E^+(X) = E^+(X_1) \cup E^+(X_2)$, and for $i \in \lbrace 1, 2 \rbrace$
and $e \in E^+(X_i)$ we have
$$\tau(\resmin{X}{\omega(e)}{\omega(e)}{e}) = \tau(\resmin{X_i}{\omega(e)}{\omega(e)}{e})$$
so it follows that
$$[\chi(\lambda(e)) \tau(\resmin{X}{\omega(e)}{\omega(e)}{e})]^+
= [\chi(\lambda(e)) \tau(\resmin{X_i}{\omega(e)}{\omega(e)}{e})]^+.$$
A dual claim holds for $e \in E^-(X_i)$ and the claim then follows
directly from the definition of $\tau$.
\end{proof}

\begin{corollary}\label{cor_tausplit}
Let $X$ be an idempotent tree with start vertex $v$. If $v$ is an
edge with $\alpha(e) = v$ then
$$\tau(X) = \tau(X \setminus \omega(e)) \ [\chi(\lambda(e)) \tau(\resmin{X}{\omega(e)}{\omega(e)}{e})]^+$$
while if $e$ is an edge with $\omega(e) = v$ then 
$$\tau(X) = \tau(X \setminus \alpha(e)) \ [\tau(\resmin{X}{\alpha(e)}{\alpha(e)}{e}) \chi(\lambda(e))]^*.$$
\end{corollary}
\begin{proof}
We prove the claim in the case that $\alpha(e) = v$, the case that
$\omega(e) = v$ being dual. Let $X_1 = X \setminus e = X \setminus \omega(e)$,
let $S$ be the set of edges in $X$ which are incident with $v$
and let $X_2 = X \setminus (S \setminus \lbrace e \rbrace)$ be the maximum
subtree of $X$ containing $e$ but none of the other edges incident
with $v$. Now clearly we have $E^+(X_2) = \lbrace e \rbrace$ and $E^-(X_2) = \emptyset$
so by the definition of $\tau$ we have
$$\tau(X_2) = [\chi(\lambda(e)) \tau(\resmin{X}{\omega(e)}{\omega(e)}{e})]^+.$$
We also have $X = X_1 \cup X_2$ and $X_1 \cap X_2 = \lbrace v \rbrace$
so by Proposition~\ref{prop_tausplit}
$$\tau(X) = \tau(X_1) \tau(X_2) = \tau(X \setminus \omega(e)) [\chi(\lambda(e)) \tau(\resmin{X}{\omega(e)}{\omega(e)}{e})]^+$$
as required.
\end{proof}

Next we define a map $\rho : UT^1(\Sigma) \to M$, from the set of
isomorphism types of (not necessarily pruned) $\Sigma$-trees to the adequate monoid $M$.
Suppose a tree $X$ has trunk vertices
$v_0, \dots, v_n$ in sequence. For $1 \leq i \leq n$ let $a_i$ be the label
of the edge from $v_{i-1}$ to $v_i$. For $0 \leq i \leq n$
let
$$X_i = \resmin{X}{v_i}{v_i}{\theta(X)}$$
that is, $X_i$ is the maximum connected subgraph of $X$ containing $v_i$ but no trunk edges,
viewed as an idempotent tree with start and end vertex $v_i$.
Then we define
$$\rho(X) \ = \ \tau(X_0) \ \chi(a_1) \ \tau(X_1) \ \chi(a_2) \ \dots \ \chi(a_{n-1}) \ \tau(X_{n-1}) \ \chi(a_n) \ \tau(X_n).$$
The value of $\rho$ clearly depends only on the isomorphism type of $X$ so
$\rho$ is indeed a well-defined map on $UT^1(\Sigma)$.
Notice also that if $X$ is an idempotent tree then $\rho(X) = \tau(X)$.
We now establish some elementary properties of the map $\rho$.

\begin{proposition}\label{prop_rhosplit}
Let $X$ be a tree with trunk vertices $v_0, \dots, v_n$ in sequence,
where $n \geq 1$. Let $a_1$ and $a_n$ be the labels of the edges from $v_0$
to $v_1$ and from $v_{n-1}$ to $v_n$ respectively. Then
$$\rho(X) = \tau(\resmin{X}{v_0}{v_0}{v_1}) \chi(a_1) \rho(\resmin{X}{v_1}{}{v_0}) = \rho(\resmin{X}{}{v_{n-1}}{v_n}) \chi(a_n) \tau(\resmin{X}{v_n}{v_n}{v_{n-1}}).$$
\end{proposition}
\begin{proof}
We prove the first equality, the remaining part being dual. Let $X_0, \dots, X_n$ be as in the definition of $\rho$, so that
$$\rho(X) = \tau(X_0) \ \chi(a_1) \ \tau(X_1) \ \chi(a_2) \ \dots \ \chi(a_{n-1}) \ \tau(X_{n-1}) \ \chi(a_n) \ \tau(X_n).$$
It follows straight from the definition that 
$$\rho(\resmin{X}{v_1}{}{v_0}) = \tau(X_1) \ \chi(a_2) \ \dots \ \chi(a_{n-1}) \ \tau(X_{n-1}) \ \chi(a_n) \ \tau(X_n)$$
so we have
\begin{align*}
\rho(X) &= \tau(X_0) \ \chi(a_1) \ \rho(\resmin{X}{v_1}{}{v_0}) \\
&= \tau(\resmin{X}{v_0}{v_0}{v_1}) \ \chi(a_1) \ \rho(\resmin{X}{v_1}{}{v_0})
\end{align*}
as required.
\end{proof}

The next proposition says that the map $\rho$ is actually a well-defined
$(2,1,1,0)$-morphism from the monoid of isomorphism types of $\Sigma$-trees
(with unpruned operations) to the
adequate monoid $M$. Later, we shall see that it even induces a well-defined
map from the monoid of isomorphism types of pruned $\Sigma$-trees (with pruned operations) to
$M$.

\begin{proposition}\label{prop_rhounprunedmorphism}
The map $\rho : UT^1(\Sigma) \to M$ is a morphism of
$(2,1,1,0)$-algebras.
\end{proposition}
\begin{proof}
Let $X$ and $Y$ be trees, say with trunk vertices $u_0, \dots, u_m$ and
$v_0, \dots, v_n$ in sequence respectively. For each $1 \leq i \leq m$ let
$a_i$ be the label of the edge from $u_{i-1}$ to $u_i$, and for each
$1 \leq i \leq n$ let $b_i$ be the label of the edge from $v_{i-1}$ to $v_i$.
For each $0 \leq i \leq m$ let
$X_i = \resmin{X}{u_i}{u_i}{\theta(X)}$
and similarly for each $0 \leq i \leq n$ define
$Y_i = \resmin{Y}{v_i}{v_i}{\theta(Y)}$.

Consider now the unpruned product $X \times Y$. It is easily seen that
for $0 \leq i < m$ we have
$$\resmin{(X \times Y)}{u_i}{u_i}{\theta(X \times Y)} = X_i$$
while for $0 < i \leq n$ we have
$$\resmin{(X \times Y)}{v_i}{v_i}{\theta(X \times Y)} = Y_i.$$
Considering now the remaining trunk vertex $u_m = v_0$ of $X \times Y$ we have
$$\resmin{(X \times Y)}{u_m}{u_m}{\theta(X \times Y)}
= \resmin{(X \times Y)}{v_0}{v_0}{\theta(X \times Y)}
= X_m \times Y_0.$$
By Proposition~\ref{prop_tausplit} and the definition of unpruned
multiplication we have $\tau(X_m \times Y_0) = \tau(X_m) \tau(Y_0)$.
So using the definition of $\rho$ we have
{\small
\begin{align*}
\rho(X \times Y) &= \tau(X_0) \chi(a_1) \tau(X_1) \dots \chi(a_m) \tau(X_m \times Y_0) \chi(b_1) \tau(Y_1) \chi(b_2) \dots \chi(b_n) \tau(Y_n) \\
&= \tau(X_0) \chi(a_1) \tau(X_1) \dots \chi(a_m) \tau(X_m)  \tau(Y_0) \chi(b_1) \tau(Y_1) \chi(b_2) \dots \chi(b_n) \tau(Y_n) \\
&= \rho(X) \rho(Y).
\end{align*}}

Next we claim that $\rho(X^{(+)}) = \rho(X)^+$. We prove this by induction
on the number of trunk edges in $X$. If $X$ has no trunk edges then
$X = X^{(+)}$ and so using the fact that $\tau(X) \in E(M)$ is fixed by
the $+$ operation in $M$ we have
$$\rho(X^{(+)}) = \rho(X) = \tau(X) = \tau(X)^+ = \rho(X)^+.$$
Now suppose for induction that $X$ has at least one trunk edge and that the claim holds for
trees with strictly fewer trunk edges. Recall that
$$X_0 = \resmin{X}{u_0}{u_0}{\theta(X)} = \resmin{X}{u_0}{u_0}{u_1}$$
and let $Z = \resmin{X}{u_1}{}{u_0}$. Now
\begin{align*}
\rho(X^{(+)}) &= \tau(X^{(+)})  &\text{ (by the definition of $\rho$)} \\
&= \tau(X_0) [\chi(a_1) \tau(Z^{(+)})]^+ &\text{ (by Corollary~\ref{cor_tausplit})} \\
&= \tau(X_0) [\chi(a_1) \rho(Z^{(+)})]^+  &\text{ (by the definition of $\rho$)} \\
&= \tau(X_0) [\chi(a_1) \rho(Z)^+]^+  &\text{ (by the inductive hypothesis)} \\
&= \tau(X_0) [\chi(a_1) \rho(Z)]^+ &\text{ (by Proposition~\ref{prop_adequatebasics}(ii))} \\ 
&= [\tau(X_0) \chi(a_1) \rho(Z)]^+ &\text{ (by Proposition~\ref{prop_adequatebasics}(iv))} \\ 
&= \rho(X)^+ &\text{ (by Proposition~\ref{prop_rhosplit})}
\end{align*}
as required.

A dual argument shows that $\rho(X^{(*)}) = \rho(X)^*$. Finally, it follows
directly from the definition that $\rho$ maps the identity element in
$UT^1(\Sigma)$ (that is, the isomorphism type of the trivial tree) to the
identity of $M$, and so is a $(2,1,1,0)$-morphism.
\end{proof}

Our next objective is to establish some technical lemmas involving the
function $\tau$.

\begin{lemma}\label{lemma_naturalorder}
Suppose $Y$ is an idempotent tree with start vertex $u$ and an
edge from $u$ to $v$ with label $a$. Let $Y' = \res{Y}{v}{v}$.
Then $\tau(Y) [\chi(a) \tau(Y')]^+ = \tau(Y)$ and
$[\tau(Y) \chi(a)]^* \tau(Y') = \tau(Y')$
in $M$.
\end{lemma}
\begin{proof}
We prove the second claim, the first being dual.
Let $Y_1 = Y \setminus v$ and $Y_2 = Y' \setminus u$.
For readability, we let
$B = \tau(Y_1)$, $C = \tau(Y_2)$ and $x = \chi(a)$, noting that
$B$ and $C$ are idempotent. Now from Corollary~\ref{cor_tausplit}
and the definition of
$\tau$ we have
$$\tau(Y) \ = \ \tau(Y_1) (\chi(a) \tau(Y_2))^+ \ = \ B (xC)^+$$
while
$$\tau(Y') \ = \ (\tau(Y_1) \chi(a))^* \tau(Y_2) \ = \ (Bx)^* C.$$
So now
\begin{align*}
[\tau(Y) \chi(a)]^* \tau(Y') &= [B (x C)^+ x]^* (Bx)^* C \\
&= [B (x C)^+ x (Bx)^* C]^* &\text{ (by Proposition~\ref{prop_adequatebasics}(iv))} \\
&= [(x C)^+ B x (Bx)^* C]^* &\text{ (since idempotents commute)} \\
&= [(x C)^+ B x C]^* &\text{ (by Proposition~\ref{prop_adequatebasics}(iii))} \\
&= [B (x C)^+ x C]^* &\text{ (since idempotents commute)} \\
&= [B x C]^* &\text{ (by Proposition~\ref{prop_adequatebasics}(iii))} \\
&= (B x)^* C &\text{ (by Proposition~\ref{prop_adequatebasics}(iv))} \\
&= \tau(Y')
\end{align*}
as required.
\end{proof}

\begin{lemma}\label{lemma_absorb}
Let $Y$ be an idempotent tree with start vertex $u$ and suppose $Y$
has an edge from $u$ to $v$ with label $a$. Let $Y' = \res{Y}{v}{v}$.
and suppose $e \in E(M)$
is an idempotent such that
$\tau(Y') e = \tau(Y')$.
Then $\tau(Y) (\chi(a) e)^+ = \tau(Y).$
\end{lemma}
\begin{proof}
First notice that by Proposition~\ref{prop_adequatebasics}(vi) we have
$[\chi(a)\tau(Y')]^+ (\chi(a)e)^+ = [\chi(a) \tau(Y')]^+$. Now we have
\begin{align*}
\tau(Y) (\chi(a) e)^+
&= \tau(Y) [\chi(a) \tau(Y')]^+ (\chi(a) e)^+   &\text{ (by Lemma~\ref{lemma_naturalorder})} \\
&= \tau(Y) [\chi(a) \tau(Y')]^+   &\text{ (by the observation above)} \\
&= \tau(Y)  &\text{ (by Lemma~\ref{lemma_naturalorder} again).}
\end{align*}
\end{proof}

\begin{lemma}\label{lemma_absorb2}
Suppose $\mu : X \to Y$ is a morphism of idempotent $\Sigma$-trees. Then
$\tau(Y) \tau(X) = \tau(Y)$.
\end{lemma}
\begin{proof}
We use induction on the number of edges in $X$. If $X$ has no
edges then we have $\tau(X) = 1$ and the result is clear.
Now suppose $X$ has at least one edge and for induction that the result holds
for trees $X$ with fewer edges. By the definition of $\tau$ we
have
{\small
\begin{align*}
\tau(X) = &\left( \prod_{e \in E^+(X)}
[\chi(\lambda(e)) \tau(\resmin{X}{\omega(e)}{\omega(e)}{e})]^+ \right)
 \left( \prod_{e \in E^-(X)}
[\tau(\resmin{X}{\alpha(e)}{\alpha(e)}{e}) \chi(\lambda(e))]^* \right).
\end{align*}}
while
{\small
\begin{align*}
\tau(Y) &= \left( \prod_{e \in E^+(Y)}
[\chi(\lambda(e)) \tau(\resmin{Y}{\omega(e)}{\omega(e)}{e})]^+ \right)
\left( \prod_{e \in E^-(Y)}
[\tau(\resmin{Y}{\alpha(e)}{\alpha(e)}{e}) \chi(\lambda(e))]^* \right).
\end{align*}}
Suppose now that $e \in E^+(X)$. Then since $\mu$ preserves endpoints and
maps the start vertex of $X$ to the start vertex of $Y$, the
edge $\mu(e)$ lies in $E^+(Y)$. We claim that the factor corresponding
to $e$ in the above expression for $\tau(X)$ is absorbed into $\tau(Y)$.
First notice that $\resmin{X}{\omega(e)}{\omega(e)}{e}$ has strictly fewer edges than $X$,
and the restriction of $\mu$ to this subtree is a morphism to
$Y' = \res{Y}{\mu(\omega(e))}{\mu(\omega(e))}$.
Hence, by the inductive hypothesis, we have
$\tau(Y') \tau(\resmin{X}{\omega(e)}{\omega(e)}{e}) = \tau(Y')$.
Now since $\mu(e)$ is an edge from $\mu(v)$ to $\mu(\omega(e))$ with label
$\lambda(e)$, by Lemma~\ref{lemma_absorb} we have
$$\tau(Y) [\chi(\lambda(e)) \tau(\resmin{X}{\omega(e)}{\omega(e)}{e})]^+ = \tau(Y)$$
as required.

A dual argument shows that factors in the product expression for $\tau(X)$ resulting
from edges in $E^-(X)$ are also absorbed into $\tau(Y)$. Thus, we have
$\tau(Y) \tau(X) = \tau(Y)$ as required.
\end{proof}

\begin{corollary}\label{cor_tausubgraph}
Let $Y$ be an idempotent tree and $X$ be a subtree of $Y$.
Then $\tau(Y) \tau(X) = \tau(Y)$.
\end{corollary}
\begin{proof}
The embedding of $X$ into $Y$ satisfies the conditions of
Lemma~\ref{lemma_absorb2}.
\end{proof}

\begin{corollary}\label{cor_tauretract}
Let $Y$ be a retract of an idempotent tree $X$. Then
$\tau(X) = \tau(Y)$.
\end{corollary}
\begin{proof}
Let $\pi : X \to X$ be a retraction of $X$ onto $Y$.
Since $\pi$ is a morphism, Lemma~\ref{lemma_absorb2} tells
us that $\tau(X) \tau(\pi(X)) = \tau(\pi(X)) = \tau(Y)$.
But since $Y = \pi(X)$ is a subtree of $X$, Corollary~\ref{cor_tausubgraph}
yields
$\tau(X) \tau(\pi(X)) = \tau(X)$.
\end{proof}

We now turn our attention to the function $\rho$.
\begin{lemma}\label{lemma_rhoabsorb}
Let $X$ be a tree with trunk vertices $v_0, \dots, v_n$ in
sequence. Then $\rho(X) = \rho(X) \tau(\res{X}{v_n}{v_n}).$
\end{lemma}
\begin{proof}
Let $X' = \res{X}{v_n}{v_n}$.
We use induction on the number of trunk edges in $X$. Clearly if $X$ has
no trunk edges then we have $X = X'$ and from the definition of $\rho$ we have 
$\rho(X) = \tau(X')$, so the claim reduces to the fact that $\tau(X')$
is idempotent. Now suppose $X$ has at least one trunk edge and that the claim
holds for $X$ with strictly fewer trunk edges. Let $Y = \resmin{X}{v_0}{v_{n-1}}{v_n}$,
let $Y' = \res{Y}{v_{n-1}}{v_{n-1}}$ and let
$X_n = \resmin{X}{v_n}{v_n}{v_{n-1}}$.
Let $a_n$ be the label
of the edge from $v_{n-1}$ to $v_n$. By Corollary~\ref{cor_tausplit}
we have
$$\tau(X') = \tau(X_n) [\tau(Y') \chi(a_n)]^*.$$
Now by Proposition~\ref{prop_rhosplit}
 we deduce that $\rho(X) = \rho(Y) \chi(a_n) \tau(X_n)$.
Also, by the inductive hypothesis we have $\rho(Y) = \rho(Y) \tau(Y')$.
Putting these observations together we have
\begin{align*}
\rho(X) \ \tau(X')
&= \left( \rho(Y) \ \chi(a_n) \ \tau(X_n) \right) \ \left( \tau(X_n) \ [\tau(Y') \chi(a_n)]^* \right) \\
&= [\rho(Y) \tau(Y')] \ \chi(a_n) \ \tau(X_n) \ [\tau(Y') \chi(a_n)]^* \\
&= \rho(Y) \ [\tau(Y') \chi(a_n)] \ [\tau(Y') \chi(a_n)]^* \ \tau(X_n) \\
&= \rho(Y) \ \tau(Y') \ \chi(a_n) \ \tau(X_n) \\
&= \rho(Y) \ \chi(a_n) \ \tau(X_n) \\
&= \rho(X)
\end{align*}
as required.
\end{proof}

\begin{lemma}\label{lemma_rho}
Let $X$ be a tree with trunk vertices $v_0, \dots, v_n$ in
sequence, and for $1 \leq i \leq n$ let $a_i$ be the label of the edge
from $v_{i-1}$ to $v_i$. Then
$$\rho(X) = \tau(\res{X}{v_0}{v_0}) \ \chi(a_1) \ \tau(\res{X}{v_1}{v_1}) \ \chi(a_2) \ \dots \ \tau(\res{X}{v_{n-1}}{v_{n-1}}) \ \chi(a_n) \ \tau(\res{X}{v_n}{v_n}).$$
\end{lemma}
\begin{proof}
For $0 \leq i \leq n$ let $X_i = \resmin{X}{v_i}{v_i}{\theta(X)}$ and let
$X^i = \res{X}{v_i}{v_i}$.
Recall that by definition we have
$$\rho(X) = \tau(X_0) \ \chi(a_1) \ \tau(X_1) \ \chi(a_2) \ \dots \ \chi(a_{n-1}) \ \tau(X_{n-1}) \ \chi(a_n) \ \tau(X_n)$$
so our task is to show that each term $\tau(X_i)$ can be replaced with
$\tau(X^i)$.

Once again we use induction on the number of trunk edges in $X$. If $X$ has no
trunk edges then we have $X^0 = X = X_0$ so that $\rho(X) = \tau(X_0) = \tau(X^0)$
and the claim holds. Now suppose $X$ has at least one trunk edge and that the
claim holds for $X$ with strictly fewer trunk edges; we call this the
\textit{outer inductive hypothesis} to distinguish it from another to
be introduced shortly. As in the previous proof, let
$Y = \resmin{X}{v_0}{v_{n-1}}{v_n}$. For $0 \leq i < n$ let
$Y^i = \res{Y}{v_i}{v_i}.$ Observe that
$Y$ has strictly fewer trunk 
edges than $X$. Now we have
\begin{align*}
\rho(X) &= \rho(X) \tau(X^n) &\text{ (by Lemma~\ref{lemma_rhoabsorb})} \\
&= \rho(Y) \chi(a_n) \tau(X_n) \tau(X^n) &\text{ (by Proposition~\ref{prop_rhosplit})} \\
&= \rho(Y) \chi(a_n) \tau(X^n) &\text{ (by Corollary~\ref{cor_tausubgraph})} \\
&= \tau(Y^0) \chi(a_1) \tau(Y^1) \chi(a_2) \dots \tau(Y^{n-1}) \chi(a_n) \tau(X^n)
\end{align*}
by the outer inductive hypothesis.
We now claim that
$$\rho(X) = \tau(Y^0) \chi(a_1) \tau(Y^1) \chi(a_2) \dots \tau(Y^{j-1}) \chi(a_j) \tau(X^j) \chi(a_{j+1}) \dots \chi(a_n) \tau(X^n).$$
for all $0 \leq j \leq n$. Having fixed the graph $X$, and hence the parameter
$n$, we prove
the claim by downward induction on the parameter $j$.
We have just proved that the claim holds for for $j = n$, which establishes
the base case. Suppose now for induction that $j < n$ and that the claim holds for greater
values of $j$; this we call the \textit{inner inductive
hypothesis}. Let $Z = \resmin{X}{v_{j+1}}{v_{j+1}}{v_j}$.
Then we have
\begin{align*}
\tau(Y^j) \chi(a_{j+1}) \tau(X^{j+1})
&= \tau(Y^j) \chi(a_{j+1}) \tau(Z) \tau(X^{j+1}) \\
&= \tau(Y^j) [\chi(a_{j+1}) \tau(Z)]^+ [\chi(a_{j+1}) \tau(Z)] \tau(X^{j+1}) \\
&= \tau(X^j) \chi(a_{j+1}) \tau(Z) \tau(X^{j+1}) \\
&= \tau(X^j) \chi(a_{j+1}) \tau(X^{j+1})
\end{align*}
where the first and fourth equalities hold by Corollary~\ref{cor_tausubgraph},
the second by Proposition~\ref{prop_adequatebasics}(iii) and the third 
by Corollary~\ref{cor_tausplit}. Now combining this with the inner
inductive hypothesis we have 
\begin{align*}
\rho(X) &= \tau(Y^0) \dots \tau(Y^{j-1}) \chi(a_j) [\tau(Y^j) \chi(a_{j+1}) \tau(X^{j+1})] \chi(a_{j+2}) \dots \chi(a_n) \tau(X^n) \\
&= \tau(Y^0) \dots \tau(Y^{j-1}) \chi(a_j) \tau(X^j) \chi(a_{j+1}) \tau(X^{j+1}) \chi(a_{j+2}) \dots \chi(a_n) \tau(X^n).
\end{align*}
This completes the inner inductive step, so the claim now follows by
induction. The case $j=0$ establishes the inductive step for the outer
induction, and hence proves the lemma.
\end{proof}

\begin{corollary}\label{cor_rhopruning}
Let $X$ be a tree. Then $\rho(X) = \rho(\ol{X})$.
\end{corollary}
\begin{proof}
Let $\pi : X \to X$ be a retraction with image $\ol{X}$. 
Suppose $X$ has trunk vertices $v_0, \dots, v_n$. For $1 \leq i \leq n$ let
$e_i$ be the edge from $v_{i-1}$ to $v_i$, and let $a_i \in \Sigma$ be the
label of $e_i$. For $0 \leq i \leq n$ let
$X_i = \resmin{X}{v_i}{v_i}{\theta(X)}$ and let $X^i = \res{X}{v_i}{v_i}$.
Now since $\pi$ is a retraction of $X$, it fixes all trunk vertices of $X$, so
it follows that $\pi$ is also a retraction of each $X^i$. Moreover for each
$i$ we clearly have $\pi(X^i) = \res{\pi(X)}{v_i}{v_i}$. Thus
\begin{align*}
\rho(X) &= \tau(X^0) \chi(a_1) \tau(X^1) \chi(a_2) \dots \tau(X^{n-1}) \chi(a_n) \tau(X^n) \\
&= \tau(\pi(X^0)) \chi(a_1) \tau(\pi(X^1)) \chi(a_2) \dots \tau(\pi(X^{n-1})) \chi(a_n) \tau(\pi(X^n)) \\
&= \tau(\res{\pi(X)}{v_0}{v_0}) \chi(a_1) \tau(\res{\pi(X)}{v_1}{v_1}) \chi(a_2) \dots \tau(\res{\pi(X)}{v_{n-1}}{v_{n-1}}) \chi(a_n) \tau(\res{\pi(X)}{v_n}{v_n}) \\
&= \rho(\pi(X)) \\
&= \rho(\ol{X})
\end{align*}
where the first and fourth equalities follow from Lemma~\ref{lemma_rho},
the second from Corollary~\ref{cor_tauretract}, the third from the definition
of the $X^i$ and the fifth from the
definition of $\pi$.
\end{proof}

Now let $\hat\rho : T^1(\Sigma) \to M$ be the restriction of $\rho$ to
the set of (isomorphism types of) pruned trees.
\begin{corollary}\label{cor_rhoprunedmorphism}
The function $\hat\rho$ is a $(2,1,1,0)$-morphism from $T^1(\Sigma)$ (with pruned
operations) to the adequate monoid $M$.
\end{corollary}
\begin{proof}
For any $X, Y \in T^1(\Sigma)$, combining Theorem~\ref{thm_morphism},
Corollary~\ref{cor_rhopruning} and
Proposition~\ref{prop_rhounprunedmorphism} yields
$$\hat\rho(XY) = \rho(XY) = \rho(\ol{X \times Y}) = \rho(X \times Y) = \rho(X) \rho(Y) = \hat\rho(X) \hat\rho(Y),$$
$$\hat\rho(X^+) = \rho(\ol{X^{(+)}}) = \rho(X^{(+)}) = \rho(X)^+ = \hat\rho(X)^+, \text{ and}$$
$$\hat\rho(X^*) = \rho(\ol{X^{(*)}}) = \rho(X^{(*)}) = \rho(X)^* = \hat\rho(X)^*.$$
Finally, that $\hat\rho$ maps the identity of $T^1(\Sigma)$ to the identity
of $M$ follows directly from the definition.
\end{proof}

We are now ready to prove our main result, which gives a concrete
characterisation of the free adequate monoid on a given generating set.
\begin{theorem}\label{thm_monoid}
Let $\Sigma$ be a set. Then $T^1(\Sigma)$ is a free object in the
quasivariety of adequate monoids, freely generated by the set
$\Sigma$ of base trees.
\end{theorem}
\begin{proof}
By Theorem~\ref{thm_adequate}, $T^1(\Sigma)$ is an adequate monoid. Now
for any adequate monoid $M$ and function $\chi : \Sigma \to M$,
define $\hat\rho : T^1(\Sigma) \to M$ as above. 
By Corollary~\ref{cor_rhoprunedmorphism}, $\hat\rho$ is a $(2,1,1,0)$-morphism, and
it is immediate from the definitions that $\hat\rho(a) = \chi(a)$
for every $a \in \Sigma$, so that $\hat\rho$ extends $\chi$.
Finally, by Proposition~\ref{prop_generators}, $\Sigma$ is a
$(2,1,1,0)$-algebra generating set for $T^1(\Sigma)$; it follows that
the morphism $\hat\rho$ is uniquely determined by its restriction to the
set $\Sigma$ of base trees, and hence is the unique morphism with the
claimed properties.
\end{proof}

Combining Theorem~\ref{thm_monoid} with Proposition~\ref{prop_monoidsemigroup} we also obtain
immediately a description of the free adequate semigroup.
\begin{theorem}\label{thm_semigroup}
Let $\Sigma$ be a set. Then the $T(\Sigma)$ is a free object in the
quasivariety of adequate semigroups, freely generated by the set
$\Sigma$ of base trees.
\end{theorem}

\section{Remarks}\label{sec_remarks}

In this section we collect together some observations on our methods,
their potential for wider application, and their connections to other work.

The ``pruning'' process (that is, computing a pruned tree $\ol{X}$ from
the unpruned tree $X$) can be more concretely realised as a process of
pruning ``branches''. Look for a ``branch'' of $X$ (a subtree which
contains all the vertices on the non-trunkward side of some given edge)
which can be ``folded'' (mapped by a retraction which fixes everything
\textit{except} the given branch) into the rest of the tree, and remove
it. Repeat this process until no such branches can be found, and one is
left with $\ol{X}$.

If one drops the requirement that a $\Sigma$-tree should have a directed
path from the start vertex to the end vertex and defines the
pruning of a tree to be the (isomorphism type of the) minimal image under
an \textit{arbitrary
morphism} instead of a retraction, one ends up with the Munn representation of
the free inverse monoid. (Edge-labelled directed trees which do not admit
non-identity morphisms are, up to isomorphism, exactly the subgraphs of the
Cayley graph of the free group on the labelling alphabet.)

As a consequence, the natural morphism of the free adequate semigroup onto 
the free ample semigroup and into the free inverse semigroup (taking $x^+$
to $xx^{-1}$ and $x^*$ to $x^{-1} x$) has a natural description as a map
of our trees onto Munn trees. Namely, one just takes the homomorphically
minimal image of the tree, and embeds it into the Cayley graph of the free
group with the start vertex at the identity.
More concretely, this can be realised by the following process.
Whenever two distinct edges have the same label and endpoint, identify 
the edges and their respective start vertices; dually whenever two edges 
have the same label and startpoint, identify the edges and their 
respective endpoints. Repeat until no such identifications are possible
(which must happen, since each identification reduces the number of
edges). This is essentially one part (since there is no deletion of spurs)
of the \textit{folding} process originally developed by Stallings \cite{Stallings83} and extensively
employed and extended to solve numerous problems in combinatorial group
theory.

Of course there is also a natural retraction of the free adequate monoid
onto the free monoid on the same generating set, which simply maps each tree
to the label of its trunk.

For the reader who prefers to think of free objects as ``word algebras''
(sets of equivalence classes of formal expressions involving the generators
and operations), the map $\hat\rho$ defined in Section~\ref{sec_fas} can of
course be used to give an explicit isomorphism from $T^1(\Sigma)$ to the
appropriate word algebra. This yields (up to some unimportant technicalities
involving the order in which idempotents are multiplied) a normal form for
elements as formal expressions. Note that the size of the expression obtained
from a tree is linear in the number of vertices and edges in the tree.

Branco, Gomes and Gould \cite{Branco09} have recently initiated the study
of free
objects in the quasivariety of \textit{left} adequate monoids, as part of
their theory of \textit{proper} left and right adequate semigroups. It
transpires that the $(2,1)$-subalgebra of $T^1(\Sigma)$ generated by the
base trees under multiplication and $+$ [respectively, multiplication
and $*$] is exactly the free left adequate [right adequate] monoid on
$\Sigma$. The proof, which is similar in spirit and outline
to Section~\ref{sec_fas} but rather different in some of the technical
details, will appear in a subsequent article \cite{K_onesidedadequate}.

The construction in Section~\ref{sec_fas} of a morphism from $T^1(\Sigma)$
to an adequate monoid $M$ depends only on the facts that $M$ is associative with
commuting idempotents, and that the $+$ and $*$ operations are idempotent
with idempotent images and satisfy the six properties given in the case of
adequate semigroups by Proposition~\ref{prop_adequatebasics}. Thus, the free
adequate semigroups will also be free objects in any category of $(2,1,1)$-algebras
which contains them and satisfy these conditions. This includes in particular the
class of \textit{Ehresmann semigroups} \cite{Lawson91}. Similar remarks
apply to free left adequate semigroups and free right adequate semigroups
(which in particular are free objects also in the classes \textit{left
Ehresmann semigroups} and \textit{right Ehresmann semigroups} respectively)
and to the corresponding classes of monoids.

The classes of monoids we have studied can all be generalised in an obvious
way to give corresponding classes of small categories. For example, there
is a natural notion of an \textit{adequate category}, the single object
instances of which are exactly the adequate monoids.
A natural extension of our methods can be used
to describe the free adequate, free left adequate and free right adequate
category generated by a given directed graph; in this case one works with directed trees in which
the vertices are labelled by vertices of the generating graph, and the edges
labelled by edges of the generating graph with a requirement that the vertex
labels match up with the edge labels in the obvious way.
Just as in the previous remark, the free left adequate category will also be
the free left Ehresmann
category. Left Ehresmann categories are generalisations of the \textit{restriction categories} studied
by Cockett and Lack \cite{Cockett02}, which in the terminology of semigroup
theory are \textit{weakly
left E-ample} categories \cite{GouldAmpleNotes}. The generalisation of our results
to categories thus relates to our main results in the same way that the description of the
free restriction category on a graph given in \cite{Cockett06}
relates to the descriptions of free ample and left ample monoids given by
Fountain, Gomes and Gould \cite{Fountain91,Fountain07}.

\section{Applications}\label{sec_applications}

In this section we show how our characterisation of the free adequate
monoids and semigroup can be applied to establish some of their basic
computational and algebraic properties.

Since pruning of a tree and testing trees for isomorphism can be done
by exhaustive search we have the following immediate corollary of our
main theorems. 
\begin{theorem}
The word problem for any finitely generated free adequate semigroup or
monoid is decidable.
\end{theorem}

The exact computational complexity of the word problem remains unclear, and is deserving of further study.
Notice that testing pruned trees for equivalence is exactly the isomorphism
problem for labelled, rooted, retract-free, directed trees. Since such trees
can be efficiently converted to expressions (via the map $\rho$ used
in Section~\ref{sec_fas}) it follows that the word problem for the free
adequate monoid is at least as hard as this problem.

We now turn our attention to some structural properties of free adequate
semigroups and monoids. Recall that the equivalence relation $\mathcal{J}$
is defined on any semigroup by
$a \mathcal{J} b$ if and only if $a$ and $b$ generate the same principal
two-sided ideal. A semigroup is called \textit{$\mathcal{J}$-trivial} if
no two elements generate the same principal two-sided ideal.

\begin{theorem}\label{thm_jtrivial}
Every free adequate semigroup or monoid is $\mathcal{J}$-trivial.
\end{theorem}
\begin{proof}
Let $\Sigma$ be a set.
By Theorems~\ref{thm_monoid} and \ref{thm_semigroup}
it will suffice to show that
$T^1(\Sigma)$ is $\mathcal{J}$-trivial.

Suppose $X$ and $Y$ are pruned $\Sigma$-trees such that $X \mathcal{J} Y$
in $T^1(\Sigma)$. Then there exist
$P, Q \in T^1(\Sigma)$ such that $Y = P X Q$.
By Theorem~\ref{thm_morphism} we have that
$Y$ is isomorphic to $\ol{P \times X \times Q}$.
By the definition of pruning there is a retraction
$P \times X \times Q \to \ol{P \times X \times Q}$. Let $\sigma : P \times X \times Q \to Y$ be
the composition of this map with an isomorphism from $\ol{P \times X \times Q}$
to $Y$, and let $\gamma : X \to Y$ be the restriction of $\sigma$ to $X$,
viewed as a subgraph of $P \times X \times Q$.

We claim first that the map $\gamma$ is in fact a morphism from $X$ to $Y$. Clearly, $\gamma$ preserves endpoints
of edges; what remains is to show that it maps the start and end vertices of
$X$ to the start and end vertices of $Y$ respectively. Since
$\sigma$ is a morphism it clearly maps trunk edges of $P \times X \times Q$
injectively to trunk edges of $Y$. Since every trunk edge of $X$ is a trunk
edge of $P \times X \times Q$ it follows that $\gamma$ maps trunk edges of
$X$ injectively to trunk edges of $Y$. Thus, we conclude that $Y$ has at
least as many trunk edges as $X$. But then by symmetry of assumption, $X$
and $Y$ have the same number of trunk edges, so $\gamma$ must map the trunk
edges of $X$ bijectively onto the trunk edges of $Y$. It follows easily
that $X$ maps the start and end vertices of $X$ to the start and end
vertices of $Y$ respectively, as required to show that $\gamma$ is a
morphism.

Now using symmetry of assumption again, we may also obtain a morphism
$\delta : Y \to X$.
Consider now the composition $\delta \gamma : X \to X$. Since this is a 
map on a finite set, it has an idempotent power, say $(\delta \gamma)^n$,
which is a retraction of $X$. Since $X$ is by assumption 
pruned, we conclude that $(\delta \gamma)^n$ is the identity map on $X$, 
and hence that $\gamma$ is injective on edges and vertices. In particular,
we see that
$Y$ has at least as many edges and vertices as $X$, and by symmetry of
assumption once again, we may conclude that $X$ has the same number of edges and
vertices as $Y$. It
follows that the injective morphism $\gamma$ is surjective,
which means that $\gamma$ is an isomorphism from $X$ to $Y$. Thus, $X$
and $Y$ represent the same element of $T^1(\Sigma)$.
\end{proof}

\begin{theorem}
No free adequate semigroup or monoid on a non-empty set is finitely generated
as a semigroup or monoid.
\end{theorem}
\begin{proof}
For each $X \in T^1(\Sigma)$, we let $\delta(X)$ be the greatest distance
(length of an undirected path) of any vertex from the trunk. Clearly
we have $\delta(\ol{X}) \leq \delta(X)$ for all $X$.
Moreover, for any pruned trees $X$ and $Y$ it follows easily from the
definition of unpruned multiplication that $\delta(X \times Y) = \max(\delta(X), \delta(Y))$,
so we have
$$\delta(XY) = \delta(\ol{X \times Y}) \leq \delta(X \times Y) = \max(\delta(X), \delta(Y)).$$

Now if $F$ is any finite set of pruned $\Sigma$-trees, then there exists
an upper bound on $\delta(X)$ for $X \in F$. It follows from the above
that this is also an upper bound on $\delta(X)$ for $X$ in the subsemigroup
of $T(\Sigma)$ [submonoid of $T^1(\Sigma)$] generated by $F$. But there are pruned trees
in $T(\Sigma)$ with vertices arbitrarily far
away from the trunk, so $F$ cannot generate the whole of these semigroups.
\end{proof}

\section*{Acknowledgements}

This research was supported by an RCUK Academic Fellowship. The author
would like to thank John Fountain and Victoria Gould for their
encouragement and advice, all the authors of \cite{Branco09} for
allowing him access to their unpublished work and work in progress, Robert
Gray for pointing out the connection between pruned trees and cores of graphs,
and Mark Lawson for alerting him to the existence of work on free restriction
categories \cite{Cockett06}.

\bibliographystyle{plain}

\def\cprime{$'$} \def\cprime{$'$} \def\cprime{$'$}

\end{document}